\documentclass{amsart}
\usepackage{pdfpages}
\usepackage{graphicx}
\usepackage{amssymb,latexsym}
\usepackage{amsmath}
\newtheorem{theorem}{Theorem}
 \newtheorem{lem}{Lemma}
 \numberwithin{equation}{section}
 \newtheorem{fac}{F}
\newtheorem{ax}{A}
\newtheorem{hax}{H}
\newtheorem{Def}{Definition}

\begin{document}

%
%
%
%
%
%
%
%
%

\title[Steiner-Lehmus and congruent medians]
 {The Steiner-Lehmus theorem and ``triangles with congruent medians are
 isosceles" hold in weak geometries}

\author{Victor Pambuccian}

\address{%
School of Mathematical and Natural Sciences (MC 2352)\\
Arizona State University - West Campus\\
P. O. Box 37100\\
Phoenix, AZ 85069-7100\\
U.S.A.}

\email{pamb@asu.edu}

\author{Horst Struve}

\address{%
Seminar f\"ur Mathematik und ihre Didaktik\\
 Universit\"at zu K\"oln\\
 Gronewaldstra\ss e 2\\
50931 K\"oln\\
Germany}

\email{h.struve@uni-koeln.de}

\author{Rolf Struve}

\address{%
SIGNAL IDUNA Gruppe\\
Joseph-Scherer-Stra\ss e 3\\
44139 Dortmund\\
Germany}

\email{rolf.struve@signal-iduna.de}





\begin{abstract}
We prove that (i) a generalization of the Steiner-Lehmus theorem due
to A. Henderson holds in Bachmann's standard ordered metric planes,
(ii) that a variant of Steiner-Lehmus holds in all metric planes,
and (iii) that the fact that a triangle with two congruent medians
is isosceles holds in Hjelmslev planes without double incidences of
characteristic $\neq 3$.
\end{abstract}
\maketitle

\section{Introduction}

The Steiner-Lehmus theorem, stating that a triangle with two
congruent interior bisectors must be isosceles, has received over
the 170 years since it was first proved in 1840 a wide variety of
proofs. Some of those provided within the first hundred years have
been surveyed in \cite[p.\ 131-134]{sim}, \cite{mac1}, \cite{mac2},
and \cite{sl}: most use results of Euclidean geometry. Several
proofs have been provided for foundational reasons, being valid in
Hilbert's absolute geometry (the geometry axiomatized by the plane
axioms of the groups I, II, and III of \cite{hil99}), the first one
being provided by Tarry \cite{tar}, the second one by Blichfeldt
\cite{bli1}, the third one in \cite[p.\ 125]{cas}, attributed to
Casey, and with the mention that H. G. Forder ``points out that this
proof is independent of the parallel postulate", and the fourth one
--- which, we are told, ``excels most by being ``absolute''" and
``came in a letter from H. G. Forder" --- in (\cite[p.\
460]{cox69}). The simplest proof among these is the one provided by
Descube in \cite{des}, and repeated, without being aware of
predecessors, by Tarry in \cite{tar} (and for congruent symmedians
in \cite{tar2}), and then in \cite[p.\ 124-125]{cas}, \cite{jme},
and \cite{h}, is valid not only in Hilbert's absolute planes, but in
more general geometries as well. While Blichfeldt's, Casey's, and
Forder's proofs rely on the free mobility property of the Hilbertian
absolute plane (the segment and angle transport axioms), Descube's
proof can be rephrased inside the geometry of a special class of
Bachmann's ordered metric planes, in which no free mobility
assumptions are made (and thus not all pairs of points need to have
a midpoint, and not all angles need to be bisectable), but in which
the foot of the perpendicular to the hypotenuse needs to lie between
the endpoints of that hypotenuse, to be referred to as {\em standard
ordered metric planes} .

On the other hand, the Steiner-Lehmus theorem has been generalized,
in the Euclidean setting, by  A. Henderson \cite[p.\ 265, 272,
Generalized Theorem (7)]{hen0} (repeated, without being aware of
\cite{hen0}, in \cite{np}, \cite{rus}, \cite{woy}, and \cite{ox}) by
replacing the requirement that two internal bisectors be congruent
by the weaker one that two internal Cevians which intersect on the
internal angle bisector of the third angle be congruent.

The purpose of this note is to present a proof, along the lines of
Descube's proof, of Henderson's generalization of the Steiner-Lehmus
theorem in an axiom system for standard ordered metric planes. Since
the statement of the generalized Steiner-Lehmus theorem presupposes
both notions of order
--- so one can meaningfully refer to ``internal bisector" (without mentioning that
the bisector is internal, the statement is false, see \cite{hen},
\cite{yze}, \cite{haj}, \cite{kh} and \cite{ah} for the generalized
version we shall prove) --- and metric notions
--- so that one can meaningfully refer to `` angle bisector'',
``congruent'' segments, and to an ``isosceles triangle''  --- the
setting of Bachmann's ordered metric planes represents the weakest
absolute geometry in which the Steiner-Lehmus theorem or its
generalization can be expected to hold.

The assumpton that the ordered metric plane be {\em standard} is
very likely not needed for the generalized Steiner-Lehmus theorem to
hold, but it is indispensable for our proof.

We will also present a short proof inside the theory of Hjelmslev
planes of the second of the ``pair of theorems'' considered in
\cite{bli2}, stating that a triangle with congruent medians must be
isosceles

\section{The axiom system for standard ordered metric planes}

For the reader's convenience, we list the axioms for ordered metric
planes, in a language with one sort of individual variables,
standing for {\em points}, and two predicates, a ternary one $Z$,
with $Z(abc)$ to be read as ``the point $b$ lies strictly between
$a$ and $c$'' ($b$ is not allowed to be equal to $a$ or to $c$) and
a quaternary one $\equiv$, with $ab\equiv cd$ to be read as ``$ab$
is congruent to $cd$''.  To improve  the readability of the axioms,
we will use the two abbreviations $\lambda$ and $L$, defined by

\begin{eqnarray*}
\lambda(abc) & :\Leftrightarrow & Z(abc)\vee Z(bca)\vee Z(cab),\\
L(abc) & :\Leftrightarrow & \lambda(abc)\vee a=b \vee b=c \vee c=a.
\end{eqnarray*}

 with $\lambda(abc)$ to be read as ``$a$, $b$, and $c$ are three
different collinear points'' and $L(abc)$ to be read as ``$a$, $b$,
and $c$ are collinear points (not necessarily different).'' Although
we have only points as variables, we will occasionally refer to {\em
lines}, with the following meaning: ``point $c$ {\em lies on the
line determined by} $a$ and $b$'' is another way of saying $L(abc)$,
and the line determined by $a$ and $b$ will be denoted by $\langle
a, b \rangle$. The axiom system for ordered metric planes consists
of the lower-dimension axiom $(\exists abc)\, \neg L(abc)$, which we
will not need in our proof, as well as the following axioms:

\begin{ax}\label{a2}
$Z(abc)\rightarrow Z(cba),$
\end{ax}
\begin{ax} \label{a2'}
$Z(abc)\rightarrow \neg Z(acb),$
\end{ax}
\begin{ax}\label{a4}
$\lambda(abc)\wedge (\lambda(abd)\vee b=d)\rightarrow
(\lambda(cda)\vee c=d),$
\end{ax}
\begin{ax} \label{a3'}
$(\forall abcde)(\exists f)\, \neg L(abc)\wedge Z(adb)\wedge \neg
L(abe)\wedge c\neq e\wedge \neg \lambda (cde)$\\
\hspace*{10mm}$\rightarrow [(Z(afc)\vee Z(bfc))\wedge
(\lambda(edf)\vee f=e)],$
\end{ax}
\begin{ax} \label{k1}
$ab\equiv pq \wedge ab\equiv rs \rightarrow pq\equiv rs$,
\end{ax}
\begin{ax} \label{k2}
$ab\equiv cc \rightarrow a=b$,
\end{ax}
\begin{ax} \label{abba}
$ab\equiv ba\wedge aa\equiv bb$,
\end{ax}
\begin{ax} \label{norV1}
 $(\forall abca'b') (\exists^{=1} c')\,  [\lambda(abc) \wedge ab \equiv a'b' \rightarrow \lambda (a'b'c')
\wedge ac \equiv a'c' \wedge bc \equiv b'c']$,
\end{ax}
\begin{ax} \label{norV2}
$\neg L(abx) \wedge L(abc) \wedge L(a'b'c') \wedge ab \equiv a'b'
 \wedge bc \equiv b'c' \wedge ac \equiv a'c'\\
\hspace*{17mm} \wedge  ax \equiv a'x' \wedge bx \equiv b'x'
\rightarrow xc \equiv x'c'$,
\end{ax}
\begin{ax} \label{norV3}
$(\forall abx) (\exists^{=1}  x')\, [\neg L(abx) \rightarrow x' \neq
x \wedge ax \equiv ax' \wedge bx \equiv bx']$,
\end{ax}
\begin{ax} \label{norV5}
$(\forall abxx') (\exists y)\, [\neg L(abx) \wedge x' \neq x \wedge
ax \equiv ax' \wedge bx \equiv bx'$\\
\hspace*{17mm} $\rightarrow L(aby) \wedge Z(xyx')]$.
\end{ax}

Axioms A\ref{a2}-A\ref{a3'} are axioms of ordered geometry,
A\ref{a3'} being the Pasch axiom. If we were to add the
lower-dimension axiom and $(\forall ab)(\exists c)\, a\neq b
\rightarrow Z(abc)$ to A\ref{a2}-A\ref{a3'}, we'd get an axiom
system for what Coxeter \cite{cox69} refers to as {\em ordered
geometry}. That we do not need the axiom stating that the order is
unending on all lines follows from the fact that, given two distinct
points $a$ and $b$, by (4.3) of \cite{sor}, there is a perpendicular
in $b$ on the line $\langle a, b \rangle$, and the reflection $c$
(which exists and is unique by A\ref{norV3}) of $a$ in that
perpendicular line is, by A\ref{norV5}, such that $Z(abc)$.

Axioms A\ref{k1}-A\ref{abba} are axioms K1-K3 of \cite{sor},
A\ref{norV1} is W1 of \cite{sor}, A\ref{norV2} is W2 of \cite{sor},
A\ref{norV3} is W3 of \cite{sor}, and A\ref{norV5} is the
conjunction of axioms W4 and WA of \cite{sor}. The reflection of
point $x$ in line $\langle a, b\rangle$, whose existence is ensured
by A\ref{norV3}, will be denoted by $\sigma_{ab}(x)$, and the
reflection of $a$ in $b$, which exists as W5 of \cite{sor} holds in
our axiom system, will be denoted by $\varrho_{b}(a)$.

That the lower-dimension axiom together with A\ref{a2}-A\ref{norV5}
form an axiom system for ordered (non-elliptic) metric planes was
shown in \cite{sor}. The theory of metric planes has been studied
intensely in \cite{bach}, where two axiom systems for it can be
found. Two other axiom systems were put forward in \cite{pambbpas}
(for the non-elliptic case only) and \cite{pambss}.

The models of ordered metric planes have been described
algebraically in the case with Euclidean metric (i.\ e.\ in planes
in which there is a rectangle) in \cite{bach3} and \cite[\S
19]{bach}, and in the case with non-Euclidean metric in
\cite{pej64}.

We will be interested in a class of ordered metric planes in which
no right angle can be enclosed within another right angle with the
same vertex, or, expressed differently, in which the foot $c$ of the
altitude $oc$ in a right triangle $oab$ (with right angle at $o$)
lies between $a$ and $b$, i.~e.\ in ordered metric planes that
satisfy

\begin{ax} \label{fusspunkt}
$Z(aoa')\wedge oa\equiv oa'\wedge ba\equiv ba'\wedge o\neq b \wedge \lambda(abc)\wedge Z(bcb')\\
\hspace*{17mm} cb\equiv cb' \wedge ob \equiv ob'\rightarrow Z(acb).$
\end{ax}

To shorten statements,  we denote the foot of the perpendicular from
$a$ to line $\langle b,c \rangle$ by $F(bca)$.

That A\ref{fusspunkt} is equivalent to the statement ${\bf RR}$,
that no right angle can be enclosed within another right angle with
the same vertex, can be seen by noticing that (i) if, assuming the
hypothesis of A\ref{fusspunkt} we have, instead of its conclusion,
that $Z(abc)$ holds, then the half-line
$\stackrel{\longrightarrow}{ox}$ with origin $o$ of the
perpendicular in $o$ to $\langle o, c\rangle$ that lies in the same
half-plane determined by $\langle o, b\rangle$ as $a$ must lie
outside of triangle $oac$ (if it lied inside it,
$\stackrel{\longrightarrow}{ox}$ would have to intersect its side
$ac$, so we would have two perpendiculars from $o$ to $\langle a,
c\rangle$, namely $\langle o, c\rangle$ and $\langle o, x\rangle$),
so the right angle $\angle aob$ is included inside the right angle
$\angle xoc$, so $\neg A\ref{fusspunkt}\Rightarrow \neg {\bf RR}$,
and (ii) if $\langle o, a\rangle\perp \langle o, d\rangle$ and
$\langle o, b\rangle\perp \langle o, c\rangle$ with
$\stackrel{\longrightarrow}{ob}$ and
$\stackrel{\longrightarrow}{oc}$ between
$\stackrel{\longrightarrow}{oa}$ and
$\stackrel{\longrightarrow}{od}$, with
$\stackrel{\longrightarrow}{oc}$ between
$\stackrel{\longrightarrow}{ob}$ and
$\stackrel{\longrightarrow}{od}$, then, with $e=F(odb)$, we must
have, by the crossbar theorem, that $\stackrel{\longrightarrow}{oc}$
intersects $eb$ in a point $p$, so that $F(obe)$ cannot lie on the
segments $ob$, for else the segments $eF(obe)$ and $op$ would
intersect, and from that intersection point there would be two
perpendiculars to  $\langle o, b\rangle$, namely $\langle e,
F(obe)\rangle$ and $\langle o, c\rangle$, so that $\neg {\bf RR}
\Rightarrow \neg A\ref{fusspunkt}$.

A\ref{fusspunkt} was first considered as an axiom in \cite[p.\
298]{bk}, in the analysis of the interplay between Sperner's
ordering functions and the orthogonality relation in affine planes.
Ordering functions satisfying A\ref{fusspunkt} are referred to in
\cite[p.\ 298]{bk} ``singled out by the orthogonality relation"
({\em durch die Orthogonalit\"at ausgezeichnet}). An axiom more general than
${\bf RR}$, stating that if an angle is enclosed within another angle with
the same vertex, then the two angles cannot be congruent, has been first considered
as axiom III,7 in \cite[p.\ 31]{ber}.

That ordered metric planes do not need to be standard, not even if
the metric is Euclidean (i.\ e.\ if there is a rectangle in the
plane), can be seen from the following example:

The point-set of the model is ${\mathbb Q}\times {\mathbb Q}$, with
the usual betwenness relation (i.~e.\ point ${\bf c}$ lies between
points ${\bf a}$ and ${\bf b}$ if and only if ${\bf c} = t{\bf a} +
(1-t){\bf b}$, with $0<t<1$, where  ${\bf a}$, ${\bf b}$, and ${\bf
c}$ are in ${\mathbb Q}\times {\mathbb Q}$ and $t$ is in ${\mathbb
Q}$) and with segment congruence $\equiv$ given by ${\bf ab}\equiv
{\bf cd}$ if and only if $\|{\bf a} - {\bf b}\|=\|{\bf c} - {\bf
d}\|$,  where $\|{\bf x}\|$ stands for $x_1^2-2x_2^2$ where ${\bf
x}=(x_1, x_2)$. Two lines $ux+vy+w=0$ and $u'x+v'y+w=0$ are
orthogonal if and only if $-2uu'+vv'=0$, and $-2$ is called the {\em
orthogonality constant} of the Euclidean plane (see \cite{bach3},
\cite{schr}, \cite{sp}, \cite{pamb4}, and  \cite{pambjsl} for more
on Euclidean planes, and \cite[p.\ 300]{bk} for a general result
regarding Sperner's ordering function that satisfy A\ref{fusspunkt}
in Euclidean planes). If ${\bf o}=(0,1)$, ${\bf a}= (1,0)$, and
${\bf b}=(2,0)$, then the $\widehat{aob}$ is a right angle, and the
foot $c=(0,0)$ of the perpendicular from $o$ to line $ab$ does not
lie between $a$ and $b$.

However, all ordered Euclidean planes with free mobility (i.\ e.\
those that can be coordinatized by Pythagorean fields) must be
standard (see \cite[p.\ 217]{bach}), as must be all absolute planes
in Hilbert's sense.

\section{Proof of the Steiner-Lehmus theorem in standard ordered metric planes}

Before starting the proof (a variant of Descube's proof) of the
Steiner-Lehmus theorem, we will first state it inside our language,
and then list without proof, the proofs being straightforward, a
series of results true in standard ordered metric planes.

With the abbreviation $M(abc)$ standing for $Z(abc)\wedge ba\equiv
bc$, to be read as ``$b$ is the midpoint of the segment $ac$'', the
generalized Steiner-Lehmus  theorem can be stated as

\begin{eqnarray} \label{sl}
 \neg L(abc)\wedge Z(amc)\wedge Z(anb)\wedge ad\equiv ab \wedge
(Z(adc)\vee Z(acd)\vee d=c)\\ \nonumber \wedge sb\equiv sd \wedge
Z(bsm)\wedge Z(csn)\wedge  bm\equiv cn\\ \nonumber \wedge
M(mpn)\wedge M(boc) \rightarrow ab\equiv ac.\\ \nonumber
\end{eqnarray}

Notice that we assume the existence of two midpoints: the midpoint
$p$ of the segment $mn$ and the midpoint $o$ of the segment $bc$.
These midpoints do exist whenever the triangle $abc$ is isosceles,
but do not have to exist in general.

{\em Notation}. Whenever the segment $bc$ has a midpoint, we define
$ab<ac$ to mean that the perpendicular bisector of  $bc$ intersects
the  open  segment $ac$. We say that the lines $\langle a, b
\rangle$ and $\langle c,d \rangle$ are perpendicular ($\langle a, b
\rangle\perp \langle c,d \rangle$) if there are two different points
$x$ and $x'$ on $\langle c,d \rangle$ such that $ax\equiv ax'$ and
$bx\equiv bx'$. We denote by $\angle xoy$ the angle formed by the
rays $\stackrel{\longrightarrow}{ox}$  and
$\stackrel{\longrightarrow}{oy}$. We say that ``$\angle abc$ is
acute'' if $\neg L(abc)$ and if $\stackrel{\longrightarrow}{bc}$
lies between   $\stackrel{\longrightarrow}{ba}$ and
$\stackrel{\longrightarrow}{bb'}$, where $\langle b, b'\rangle\perp
\langle b, a\rangle$, with $b'$ on the same side of $\langle a,
b\rangle$ as $c$. A point $p$ lies in the {\em interior} of $\angle
xoy$ if $\stackrel{\longrightarrow}{op}$ intersects the open segment
$xy$. We say that a segment $ab$ can be {\em transported}  from $c$
othe ray $\stackrel{\longrightarrow}{cd}$, if there is a point  $x$
on $\stackrel{\longrightarrow}{cd}$
(to be referred as the second endpoint of the transported segment) such that $ab\equiv cx$.\\

The facts that we will need for its proof, which will be stated
without proof, their proofs being either well known or
straightforward, are:

\begin{fac} \label{(F1)} If $\neg L(abu)$, $ba\equiv bu$, $m\neq b$, and $ma\equiv mu$,
then, for every point $p$ with $L(bmp)$, we have $px\equiv py$,
where $x=F(bap)$ and $y=F(bup)$  (points on the angle bisector of an
angle are equidistant from the legs of the angle). Also, $bx\equiv
by$.
\end{fac}


\begin{fac} \label{(F4)} If $\angle bac$ is acute, then $\angle cab$ is acute as well, and any angle with the same vertex $a$ and inside  $\angle bac$
is also acute. If the triangles $abc$ and $a'b'c'$ are congruent
(i.~e.\ $ab\equiv a'b'$, $bc\equiv b'c'$, and $ca\equiv c'a'$), and
$\angle bac$ is acute, then so is $\angle b'a'c'$.
\end{fac}

\begin{fac} \label{(F5)}  If $a$ and $a'$ have a midpoint, $Z(abc)$, $Z(a'b'c')$,
$ac\equiv a'c'$, $ab\equiv a'b'$, then $bc\equiv b'c'$ (a special
case of Euclid's Common Notion III ``If equals are subtracted from
equals, then the remainders are equal.")
\end{fac}


\begin{fac} \label{(F6)}  The betweenness relation is preserved under orthogonal
projection, i.\ e.\ if $Z(oa'b')$, $L(oab)$, and $\langle a,
a'\rangle$ and $\langle b, b'\rangle$ are perpendicular on $\langle
a, b \rangle$, then $Z(oab)$.
\end{fac}

\begin{fac} \label{(F7)}
The base angles of an isosceles triangles are  acute, i.~e.\ if
$\neg L(abc)$, $ab\equiv ac$, $\langle b, b'\rangle \perp \langle b,
c\rangle$ and $\langle c, c'\rangle \perp \langle c, b\rangle$, and
$b'$ and $c'$ lie on the same side of $\langle b, c\rangle$ as $a$,
then $\stackrel{\longrightarrow}{ba}$ lies between
$\stackrel{\longrightarrow}{bc}$ and
$\stackrel{\longrightarrow}{bb'}$ and
$\stackrel{\longrightarrow}{ca}$ lies between
$\stackrel{\longrightarrow}{cb}$ and
$\stackrel{\longrightarrow}{cc'}$.
\end{fac}

\begin{fac} \label{(F8)}   If $ab\equiv a'b'$, $ac\equiv a'c'$, $Z(abc)$, and
$Z(a'b'c')\vee Z(a'c'b')$, then $Z(a'b'c')$.
\end{fac}


\begin{fac} \label{(F10)} If $\angle bac$ has an interior angle bisector, then any
segment $xy$ on $\langle a, b\rangle$ can be transported on any
halfline that is included in $\langle a,c\rangle$, i.~e.\ for all
$x, y$ with $x\neq y$, $L(abx)$ and $L(aby)$, and any $u, v$ with
$L(acu)$ and $L(acv)$ and $u\neq v$, there exists a $z$ with
$uz\equiv xy$ and $Z(vuz)$.
\end{fac}

\begin{fac} \label{(F11)} If $ab\equiv a'b'$, with $a\neq b$, and if the lines $\langle
a, b\rangle$ and $\langle a', b'\rangle$ intersect, then the two
lines have an angle bisector (as both segments $ab$ and $a'b'$ can
be transported along the lines $\langle a, b\rangle$ and $\langle
a', b'\rangle$ respectively (via A\ref{norV1}) to their intersection
point).
\end{fac}

\begin{fac} \label{(F12)} If $\neg L(abc)$, $ab\equiv a'b'$, $ac\equiv a'c'$, $cb \equiv c'b'$, $L(adc)$,
$ad\equiv a'd'$, $dc\equiv d'c'$, then $dF(abd)\equiv d'F(a'b'd')$.
\end{fac}

\begin{fac} \label{(F13)} If $p_1$ and $p_2$ are two distinct points on the same side of the line
$\langle a, b\rangle$, and $p_1F(abp_1)\equiv p_2F(abp_2)$, then the
lines $\langle p_1, p_2\rangle$ and $\langle a, b\rangle$ do not
meet (given that the perpendicular from the point of intersection of
segments $p_1F(abp_2)$ and $p_2F(abp_1)$ to $\langle a, b\rangle$ is
perpendicular to $\langle p_1, p_2\rangle$  as well).
\end{fac}

We will also need the following lemmas:

\begin{lem} \label{l1}
$\neg L(bsc)\wedge Z(bxs)\wedge
Z(cxm')\wedge Z(csn)$\\
$\wedge cn\equiv cm' \rightarrow bm'<bn$
\end{lem}

\begin{proof}
Since $cn\equiv cm'$,   $d=F(nm'c)$  is the midpoint of the segment
$m'n$, and thus $\langle d, c\rangle$ is the perpendicular bisector
of the segment $m'n$. By the crossbar theorem the ray
$\stackrel{\longrightarrow}{cd}$ intersects the segment $xs$ in a
point $p$. By the Pasch theorem applied to $\Delta bsn$ and secant
$\langle d, p\rangle$, we get that the latter must intersect the
segment $bn$, thus $bm<bn$.
\end{proof}



\begin{lem} \label{l2}
$M(boc)\wedge \neg L(abc)\wedge Z(ab'c)\wedge ab\equiv
ab'\rightarrow ab<ac.$
\end{lem}

\begin{proof}
Let $o'=F(bb'a)$, i.\ e.\ the midpoint of segment $bb'$. Since
$\langle o, o'\rangle$ and $\langle c, b'\rangle$ have a common
perpendicular (by \cite[\S4,2, Satz 2]{bach}), they cannot
intersect, and thus, by the Pasch axiom, $\langle a, o'\rangle$
intersects the side $bc$ of $\triangle bcb'$ in a point $p$ with
$Z(ao'p)$ and $Z(bpc)$. Point $p$ cannot coincide with $o$, as
$\langle o, o'\rangle$ and $\langle c, b'\rangle$ do not intersect,
and it also cannot be such that $Z(cpo)$, for in that case the Pasch
axiom with $\Delta pac$ and secant $\langle o, o'\rangle$ would ask
the latter to intersect segment $ac$, contradicting the fact that
$\langle o, o'\rangle$ and $\langle a, c\rangle$ do not intersect.
Thus $Z(opb)$ must hold. Also by the Pasch axiom the perpendicular
bisector of segment $bc$ must intersect one of the segments $ac$ or
$ab$. Suppose it intersects segment $ab$ in $s$.  By
A\ref{fusspunkt}, with $f=F(pbo')$, we have $Z(pfb)$. Since we have
$Z(po'a)$ and $Z(asb)$, for $g=F(pba)$ and for   $o$, which is
$F(pbs)$, we have, by  F\ref{(F6)}, $Z(pfg)$ and $Z(gob)$. From
these two betweenness relations and $Z(pfb)$ we get that $Z(opb)$
cannot hold, a contradiction. Hence the perpendicular bisector of
$bc$ must intersect $ac$.
\end{proof}

\begin{lem} \label{l3}
$M(bab')\wedge o\neq a\wedge Z(abc)\wedge ob\equiv ob'\wedge
od\equiv ob \wedge (Z(odc)\vee Z(ocd))\rightarrow Z(odc)$.
\end{lem}

\begin{proof}
Suppose we have $Z(odc)$. Let $x$ be the intersection point of the
perpendicular in $b$ on $\langle b, a\rangle$ with side $oc$ of
$\Delta aoc$ (which it must intersect by the Pasch axiom and the
fact that it cannot intersect $\langle o, a\rangle$), and let $y$ be
the point of intersection of the perpendicular in $b$ on $\langle b,
d\rangle$ with segment $xd$ (by ${\bf RR}$, i.\ e.\ by
A\ref{fusspunkt}, there must be such an intersection point). Let
$m=F(bdo)$.  Since we have $Z(dmb)$, we should, by F\ref{(F6)}, also
have $Z(doy)$, which cannot be the case, since we have $Z(dyx)$ and
$Z(dxo)$, thus $Z(dyo)$. Hence we must have $Z(odc)$.
\end{proof}

\begin{lem} \label{l4}
If $\neg L(oac)$, ray $\stackrel{\longrightarrow}{ob}$ lies between
$\stackrel{\longrightarrow}{oa}$ and
$\stackrel{\longrightarrow}{oc}$, then one of the halflines
determined by $o$ on $\langle o, d\rangle$, the line for which
$\sigma_{oc}\sigma_{ob}\sigma_{oa}=\sigma_{od}$ (which exists, since
metric planes satisfy the three reflection theorem for concurrent
lines), also lies between $\stackrel{\longrightarrow}{oa}$ and
$\stackrel{\longrightarrow}{oc}$.
\end{lem}

\begin{proof}
Let $a'=\sigma_{oF(oba)}(a)$. If $a'$ lies on the side determined by
$\langle o, c\rangle$ opposite to the one in which $a$ lies, then
segmenmt $aa'$ must intersect $\stackrel{\longrightarrow}{oc}$ in a
point $z$ and thus we have, with $x=a$ and $y=F(oba)$, that $x$,
$y$, and $z$ are three points on the rays
$\stackrel{\longrightarrow}{oa}$, $\stackrel{\longrightarrow}{ob}$,
and $\stackrel{\longrightarrow}{oc}$ respectively, such that
$Z(xyz)$ and $y=F(xzo)$. Three such points can be found even in case
$a'$ lies on the same side determined by $\langle o, c\rangle$ as
$a$. For, in that case, it must be that $c'=\sigma_{ob}(c)$  is on
the  side determined by $\langle o, a\rangle$ opposite to the one in
which $c$ lies. To see this, notice that, by the crossbar theorem,
$\stackrel{\longrightarrow}{oa'}$ must intersect the segment
$cF(obc)$ in a point $p$, so we have $Z(F(obc)pc)$. Since
$\sigma_{ob}$ preserves betweenness and $p'=\sigma_{ob}(p)$ is on
$\stackrel{\longrightarrow}{oa}$, we have $Z(F(obc)p'c')$, i.\ e.\
segment $cc'$ intersects $\langle o, a\rangle$. So, in this case, we
set $x=p'$, $y=F(obc)$, $z=c$, to have three points on the rays
$\stackrel{\longrightarrow}{oa}$, $\stackrel{\longrightarrow}{ob}$,
and $\stackrel{\longrightarrow}{oc}$ respectively, such that
$Z(xyz)$ and $y=F(xzo)$.

Let now $z' = \sigma_{oc}\sigma_{ob}\sigma_{oa}(z)=
\sigma_{oc}\sigma_{ob}(z)$. Note that, with $d=F(zz'o)$, we have
$\sigma_{oc}\sigma_{ob}\sigma_{oa}= \sigma_{od}$, so if we prove
that $z'$ lies inside the angle $\angle aoc$, we are done, since
$d$, as the midpoint of $zz'$, must lie inside the angle $\angle
aoc$ as well.

With  $u=\sigma_{ob}(z)$, we notice that we must have one of
$Z(yxu)$ or $Z(yzu)$ (for, if $Z(yux)$, then, by the fact that
reflections in lines preserve  betweenness, we must have  $Z(yzu)$),
so, given that $\sigma_{oc}\sigma_{ob}\sigma_{oa}=
\sigma_{oa}\sigma_{ob}\sigma_{oc}$, we may assume, w.\ l.\ o.\ g.\
that that $Z(yxu)$. With $v=F(oau)$ and $w=F(oaF(obz)$, we must
have, by F\ref{(F6)}, $Z(wxv)$. By A\ref{fusspunkt} we also have
$Z(owx)$, thus also $Z(oxv)$. Since the line $\langle u, v\rangle$
intersects the extensions of two sides of $\Delta oxz$, it cannot,
by the Pasch axiom, intersect the segment $oz$, so if the segment
$uz'$ intersects line $\langle o, z\rangle$, then it can intersect
it only in a point $q$ with $Z(ozq)$ (and $Z(uqz')$). In that case,
by the Pasch axiom, the secant $\langle o, d\rangle$ must intersect
the side $qz'$ of $\Delta zqz'$ in a point $r$. The perpendicular in
$r$ on $\langle q, z'\rangle$ must intersect, by the Pasch axiom,
one of the sides $oq$ or $oz'$ of $\Delta oqz'$. It cannot intersect
$oq$, for then, it would also have to intersect, by the Pasch axiom,
side $vq$ of $\Delta oqv$, and from that intersection point there
would be two perpendiculars to $\langle q, z'\rangle$. So it must
intersect segment $oz'$ in $s$. By the Pasch axiom applied to
$\Delta doz'$ and secant $\langle r, s\rangle$, we conclude that
there is a point $f$ with $Z(rfs)$ and $Z(dfz')$. By
A\ref{fusspunkt} we have $Z(rF(rfd)f)$, and the Pasch axiom applied
to $\Delta rz'f$ with secant $\langle d, F(rfd)\rangle$ gives a
point of intersection of the latter with segment $z'r$, a
contradiction, as from that point one has dropped two distinct
perpendiculars to $\langle r, s\rangle$. Thus, $z'$ has to lie
inside the angle $\angle aoc$, and we are done.
\end{proof}


\begin{theorem}
The generalized Steiner-Lehmus theorem, (\ref{sl}), holds in
Bachmann's standard ordered metric planes.
\end{theorem}

\begin{proof}

Let $b'=\varrho_p(b)$. Then, since $\varrho_p$ is an isometry,
$nb\equiv mb'$ and $bm\equiv b'n$. Since $bm\equiv cn$, we have
$nb'\equiv nc$ as well.

We will first show that $Z(aF(abs)b$ and $Z(aF(acs)c)$ must hold.
The perpendicular raised in $s$ on $\langle a, s\rangle$ must
intersect, by the Pasch axiom, one of the sides $ac$ or $an$ of
$\Delta acn$. Thus, it must intersect at least one of the sides $ab$
and $ac$ of $\Delta abc$. If it intersects both sides (including the
ends $b$ and $c$ of the segments $ab$ and $ac$), then, by
A\ref{fusspunkt}, we get the desired conclusion, namely that
$Z(aF(abs)b$ and $Z(aF(acs)c)$. Suppose the perpendicular raised in
$s$ on $\langle a, s\rangle$ intersects one of the two, say $ac$,
but {\em not}  the closed segment $ab$. Since the point $q$,
obtained by reflecting in $\langle a, s\rangle$ the intersection of
the perpendicular raised in $s$ on $\langle a, s\rangle$ with $ac$,
lies on both $\langle a, s\rangle$ and on $\langle a, b\rangle$, we
must have $Z(abq)$ (since we assumed that we do not have $Z(aqb)$)
and $\langle s, a\rangle \perp \langle s, q\rangle$. We want to show
that we still need to have $Z(aF(abs)b$ in this case as well.
Suppose that were not the case, and we'd have $Z(F(abs)ba)$. We thus
have $Z(F(abs)bn)$ and $Z(cmF(acs))$. With $s'=\varrho_p(s)$, we
have, given that point-reflections are isometries, $s'n\equiv sm$,
$s'b'\equiv sb$, and $bm\equiv b'n$. Simce $bm\equiv cn$, we also
have $cn\equiv b'n$. With $u=F(b'cn)$, $w=\varrho_{nu}(s)$, we
notice that $Z(nwb')$, $ns\equiv nw$, $sc\equiv wb'$  (since lines
$\langle n, c\rangle$ and $\langle n, b'\rangle$ are symmetric with
respect to $\langle n, u\rangle$, and symmetry in lines preserves
both congruence and betweenness). Let $v$ be the point (whose
existence is ensured by A\ref{norV1}) for which  (by F\ref{(F8)})
$Z(nvb')$, $ns'\equiv b'v$, and $b's'\equiv nv$. We thus have
$bs\equiv nv$, $ns\equiv nw$, $sm\equiv b'v$,  $sc\equiv wb'$. By
Lemma \ref{l3} (with $(s, F(abs), b, n)$ and $(s, F(acs), m,c)$ for
$(o,a,b,c)$), using F\ref{(F11)} (i.\ e.\ bearing in mind that $sb$
can be transported from $n$ on $\stackrel{\longrightarrow}{ns}$, and
that $sm$ can be transported from $c$ on
$\stackrel{\longrightarrow}{cs}$, and that the second points
resulting from the trasnport are on the open segments $ns$ and
$cs$), we deduce that $Z(nvw)$ and $Z(b'vw)$, which is impossible.
This proves that both $Z(aF(abs)b)$ and $Z(aF(acs)c)$ must hold.

If  $o$ coincides with $F(bcs)$, then $sb\equiv sc$ and thus, since
$bm\equiv cn$, also $sm\equiv sn$ (by F\ref{(F5)}). Since
$\sigma_{so}(b)=c$, and $\sigma_{so}$ is an isometry and preserves
betweenness, we must, by the uniqueness requirement in A\ref{norV1},
have $\sigma_{so}(m)=n$, and thus $\sigma_{so}$ maps line $\langle
b, n\rangle$ onto line $\langle m, c\rangle$. Since these two lines
intersect in $a$, point $a$ must lie on the axis of reflection, and
thus $ab\equiv ac$.\\

Suppose $o\neq F(bcs)$. W. l.\ o.\ g.\ we may assume that
$Z(boF(bcs))$. By the Pasch axiom, the perpendicular bisector of the
segment $bc$ must intersect one of the sides $sb$ and $sc$ of
$\triangle sbc$. Given $Z(boF(bcs))$ and F\ref{(F6)},  it must
intersect side $sb$ in a point $x$ (and thus does not intersect the
side $sc$). By the Pasch axiom, line $\langle x, o\rangle$ must
intersect one of the sides $ab$ and $ac$ of $\triangle abc$. Line
$\langle x, o\rangle$ cannot intersect $ac$, for else, by the Pasch
axiom applied to $\triangle asc$ and secant $\langle x, o\rangle$,
it would have to intersect one of the sides $sa$ and $sc$. Since we
have already seen that $\langle x, o\rangle$ cannot intersect
segment $sc$,  $\langle x, o\rangle$ must intersect the segment $sa$
in a point $z$. We will show that this leads to a contradiction. Let
$z_1=F(abz)$ and $z_2=F(acz)$. We have shown that $Z(aF(abs)b)$ and
$Z(aF(acs)c)$, so, given $Z(sza)$, we can apply F\ref{(F6)},  to
obtain then have $Z(az_1b)$, $Z(az_2c)$. We also have $az_1\equiv
az_2$ (by F\ref{(F1)}) and $zb\equiv zc$ (as $z$ is a point on the
perpendicular bisector of segment $bc$). Since $\sigma_{az}$ maps
line $\langle a, b\rangle$ onto line $\langle a, c\rangle$, we have
that $zb\equiv z\sigma_{az}(b)$, and thus $zc\equiv
z\sigma_{az}(b)$, and $L(ac\sigma_{az}(b))$. Since
$\sigma_{az}(b)\neq c$ (else, we'd have $ab\equiv ac$, so
$o=F(bcs)$), we must have $\sigma_{az}(b)=\varrho_{z_2}(c)$, a
contradiction, as $\sigma_{az}$ preserves the betweenness relation,
and we have $Z(az_1b)$, and thus should have
$Z(az_2\sigma_{az}(b))$. Thus $\langle x, o\rangle$ must intersect
$ab$ in point $g$.

Let $m'=\sigma_{ox}(m)$. Since $c=\sigma_{ox}(b)$ and $\sigma_{ox}$
is an isometry, we have $cm\equiv bm'$, as well as $bm\equiv cm'$,
and thus, given the hypothesis that $bm\equiv cn$, we have
$cm'\equiv cn$, and since $x$ is a fixed point of $\sigma_{ox}$ and
$Z(bxm)$, we have $Z(cxm')$, and thus the hypothesis of  Lemma
\ref{l1} holds, and thus so must the conclusion, i.~ e.\  $bm'<bn$.
By F\ref{(F10)}, $cm$ can be transported from $n$ on the ray
$\stackrel{\longrightarrow}{nb}$ to get $m_1$, which, by Lemma
\ref{l2} (which can be applied, as $m'n$ does have a midpoint, given
that the other two sides of $\triangle mnm'$ have midpoints,  see
\cite[\S4,2, Satz 2]{bach}), must be such that $Z(bm_1n)$ and
$nm_1\equiv cm$. Since reflections in points are isometries and
preserve betweenness, for $m_2=\varrho_{p}(m_1)$ we have $mm_2\equiv
nm_1$ (thus $mm_2\equiv cm$) and $Z(mm_2b')$, so, by Lemma \ref{l2}
(which can be applied as the segment $cb'$ does have a midpoint, as
the two other sides of $\triangle bcb'$ have midpoints,  see
\cite[\S4,2, Satz 2]{bach}),
we have $mc< mb'$. \\

Let $h$ be the intersection point of the perpendicular bisector of
$b'c$ with segment $mb'$. Since $nc\equiv nb'$, we have
$L(nhF(b'cn))$. Let $a'=\varrho_p(a)$ and $a_1=\sigma_{nh}(a')$. We
have $na_1\equiv na'$ and $na'\equiv am$ (since symmetries in both
lines and points are isometries),  thus $na_1\equiv ma$. Let
$m''=\sigma_{as}(m)$ and $b_1=\sigma_{as}(b)$. Given $ac<ab$, we
must have $Z(acb_1)$ (by Lemma \ref{l2}), and thus, by the Pasch
axiom applied to $\triangle anc$ and secant $\langle b_1, s\rangle$,
the latter must intersect side $na$, and the point of intersection
is $m''$, so $Z(am''n)$. Since $ma\equiv m''a$, we also have
$na_1\equiv m''a$. We also have $b'a'\equiv ba$, $b'a'\equiv ca_1$
(since symmetries in both points  and  lines are isometries), so
$ba\equiv ca_1$, and, since $ba\equiv b_1a$, also $ca_1\equiv b_1a$.

We turn our attention to the congruent triangles $ca_1n$ and
$b_1am''$. The $\angle aca_1$ being bisectable (by F\ref{(F11)}),
let $\stackrel{\longrightarrow}{cc_1}$ be its internal bisector
(i.~e.\ $\stackrel{\longrightarrow}{cc_1}$  lies between
$\stackrel{\longrightarrow}{ca}$  and
$\stackrel{\longrightarrow}{ca_1}$) let
$\varphi=\sigma_{ca}\sigma_{cc_1}$. By F\ref{(F4)}, $\angle
\varphi(a_1)c\varphi(n)$ is acute (as triangles $ca_1n$ and
$c\varphi(a_1)\varphi(n)$ are congruent and $\angle a_1cn$ is acute
(by F\ref{(F7)}, as it is included in the base angle of the
isosceles triangle $ncb'$)). By Lemma \ref{l4}, and the fact that
$\angle \varphi(a_1)c\varphi(n)$ is acute,
$\stackrel{\longrightarrow}{c\varphi(n)}$ is between
$\stackrel{\longrightarrow}{cn}$ and
$\stackrel{\longrightarrow}{cb_1}$. Thus
$\stackrel{\longrightarrow}{c\varphi(n)}$ must intersect segment
$b_1s$ in a point $p_1$. Let $p_2$ be a point on
$\stackrel{\longrightarrow}{cp_1}$ with $cp_2\equiv b_1p_1$   (such
a point exists by F\ref{(F11)}). Notice that $p_1\neq p_2$, given
that one of the base angles of  $\triangle p_1b_1c$, $\angle
b_1cp_1$ is not acute (since its supplement,  $\angle p_1ca$ is
acute), so $\triangle p_1b_1c$ cannot be isosceles by F\ref{(F7)}.
On ray $\stackrel{\longrightarrow}{cp_1}$ there are thus two points,
$p_1$ and $p_2$, whose distance to line $\langle b_1, a\rangle$ is
the same (by F\ref{(F12)}), contradicting F\ref{(F13)}.

\end{proof}

\section{A triangle with two congruent medians is
isosceles}\label{med}

We will turn to the proof of the second result proved in \cite{bli2}
to be true in Hilbert's absolute planes, i.~e.

\begin{theorem} \label{th2}
A triangle with two congruent medians is isosceles.
\end{theorem}

We will show that this theorem is true in a purely metric setting
(without introducing a relation of order). The axiom system for this
theory can be expressed in first order logic, as done in
\cite{stru}. Here we will present it in its group-theoretical
formulation of F. Bachmann \cite[p.\ 20]{bachje}.

\emph{Basic assumption.} Let $G$ be a group which is generated by an
invariant set $S$ of involutory elements.

\emph{Notation:} The elements of $G$ will be denoted by lowercase
Greek letters, its identity by $1$, those of $S$ will be denoted by
lowercase Latin letters. The set of involutory elements of $S^{2}$
will be denoted by $P$ and their elements by uppercase letters $A,
B, ...$. The \textquoteleft stroke relation' $\alpha \mid \beta$ is
an abbreviation for the statement that $\alpha, \beta$ and
$\alpha\beta$ are involutory elements. The statement $\alpha, \beta
\mid \delta$ is an abbreviation of $\alpha \mid \delta$ and $\beta
\mid \delta$. We denote $\alpha^{-1} \sigma \alpha$ by
$\sigma^{\alpha}$.

$(G, S, P)$ is called a \textit{Hjelmslev group without double
incidences} if it satisfies
the following axioms:\\

\begin{hax}  \label{h1}
For $A, b$ there exists c with $A, b \mid c$.
\end{hax}

\begin{hax}  \label{h2}
 If $A, B \mid c, d$ then $A = B$ or $c = d$.
\end{hax}

\begin{hax}  \label{h3}
 If $a, b, c \mid e$ then $abc \in S$.
 \end{hax}

\begin{hax}  \label{h4}
 If $a, b, c \mid E$ then $abc \in S$.
  \end{hax}

\begin{hax}  \label{h5}
 There exist $a, b$ with $a \mid b$.
   \end{hax}

The elements of $S$ can be thought of as reflections in lines (and
can be thought of as lines), those of $P$ as reflections in points
(and can be thought of as points), thus $a|b$ can be read as ``the
lines $a$ and $b$ are orthogonal", $A|b$ as ``$A$ is incident with
$b$". Thus, the axioms state that: through any point to any line
there is a perpendicular, two points have at most one joining line,
the three reflection theorem for three lines incident with a point
or having a common perpendicular, stating that the composition of
three reflections in lines which are either incident with a point or
have a common perpendicular is a reflection in a line, and the
existence of a point. In contrast to Bachmann's metric planes
\cite{bach} there may be points which have no joining line.

A notion of congruence for segments can be introduced in the
following way:

\begin{Def}
$AB$ and $CD$ are called \textit{congruent} ($AB \equiv CD$) if
there is a motion $\alpha$ (i.\ e.\ $alpha\in G$) with $A^{\alpha} =
C$ and $B^{\alpha} = D$ or with $A^{\alpha} = D$ and $B^{\alpha} =
C$.
\end{Def}

Our theorem on triangles with congruent medians can be stated in
this setting as:

\begin{enumerate}
\item[(*)] \emph{Let $A, B, C$ be three non-collinear points and $C^{U} = B$ and $B^{W} = A$ and $b \,|\, A, C$ and $n \,|\,
C, W$. If $AU \equiv CW$ then there exists a line $v$ through $B$
with $A^{v} = C$, i.e.  triangle $ABC$ is isosceles.}
\end{enumerate}

This statement does not hold in Bachmann's metric planes which are
elliptic or of characteristic 3 (such as the Euclidean plane over
$GF(3)$; see  \cite[7.4]{bachje}). A Hjelmslev group $(G, S, P)$
(and with it a Bachmann plane) is called \emph{non-elliptic} if $S
\cap P = \emptyset$, and  of \emph{characteristic} $\neq 3$ if
$(AB)^{3} \neq 1$ for all $A\neq B$,

\begin{theorem} \label{medians}
Let $(G, S, P)$ be a Hjelmslev group without double incidences which
is non-elliptic and of characteristic $\neq 3$. Let $A, B, C$ be
three non-collinear points and $C^{U} = B$ and $B^{W} = A$ and $b
\,|\, A, C$ and $n \,|\, C, W$. If $AU \equiv CW$ then there exists
a line $v$ through $B$ with $A^{v} = C$.
\end{theorem}

For the proof of Theorem \ref{medians} we need the following:

\begin{lem} \label{lmed}
Let $C, W$ be points and $s, n$ lines with $C, W| n$ and $W| s$ and
$C\nmid s$. Then there exists at most one point $V \neq W$ with $V|
s$ and $CW \equiv CV$.
\end{lem}

\begin{proof}
Let $C, W \,|\, n$ and $V, W \,|\, s$ with $C\nmid s$. If $CW \equiv
CV$ then there exists a motion $\alpha$ with $C^{\alpha} = C$ and
$W^{\alpha} = V$ or with $C^{\alpha} = V$ and $W^{\alpha} = C$.

Suppose $C^{\alpha} = C$ and $W^{\alpha} = V$ (case 1). Then
$\alpha$ is a line through $C$ or a rotation which leaves $C$ fixed.
Since in the latter case $n \alpha$ is a line through $C$ (see
\cite[Section 3.4]{bachje}), we can assume that $\alpha$ is a line
$g$ which leaves $C$ fixed. Let $h$ be the line with $h| W, g$. Then
$V, W| h, s$ and according to H\ref{h2} it is $h = s$. Hence $g$ is
the unique perpendicular with $g| C, s$ and $V$ is the unique point
with $V = W^{g}$ and $CW \equiv C^{g}W^{g}\equiv CV$.


Suppose now $C^{\alpha} = V$ and $W^{\alpha} = C$ (case 2). Since
$\alpha$ or $n \alpha$ is a glide reflection (according to
\cite[Proposition 3.2]{bachje}) we can assume without loss of
generality that $\alpha$ is glide reflection i.e. $\alpha \in PS$.

Let $M$ be the midpoint of $C$ and $W$ and let $N$ be the midpoint
of $C$ and $V$ (which exist according to \cite[Proposition
2.33]{bachje}). Let $a$ be the axis of the glide reflection
$\alpha$. According to \cite[Proposition 2.32]{bachje}we have $a| M,
N$.

Since $\alpha \in PS$ there exists a line $b$ with $\alpha = bN$ and
$b| a$ (see \cite[Section 2.3]{bachje}). Hence $V = C^{\alpha} =
C^{bN}$ and $C^{b} = V^{N} = C$ (since $N$ is the midpoint of $V$
and $C$). Thus we get $b| C, a$. In an analogous way there exists a
line $d$ with $\alpha = Md$ and $d| a$. Hence $C = W^{\alpha} =
W^{Md} = C^{d}$ and $d| C, a$. Since in a non-elliptic Hjelmslev
group there is at most one perpendicular from $C$ to $a$ we get $b =
d$.

Hence $\alpha = Mb = bN$ and $M^{b} = N$, i.e. $b$ is the midline of
$M$ and $N$. Thus $W^{b} = (C^{M})^{b} = M^{b} C^{b} M^{b} = NCN =
C^{N} = V$, i.e. $b$ is also the midline of $W$ and $V$. Hence $b$
is the unique perpendicular from $C$ to $s$ (the joining line of $V,
W$) and $V$ the unique point with $V = W^{b}$.
\end{proof}

We now turn to the proof of Theorem \ref{medians}.

\begin{proof}
Let $U$ and $W$ be the midpoints of the sides $BC$ and $BA$ of
triangle $ABC$ and $AU\equiv CW$. Let $n$ and $b$ denote the lines
$\langle C, W\rangle$ and $\langle A, C\rangle$ respectively. Since
$A^{WU} = C$ there exists a midpoint V of $A, C$ (see
\cite[Proposition 2.33]{bachje})  and a midline $v = V b$ of $A, C$.
Moreover according to \cite[Proposition 2.48]{bachje}  there is a
joining line $s$ of $U, W$ which is orthogonal to $v$, i.e. $v| b,
s$. Since $v, W, U| s$ the element $vWU = h$ is a line with $h| s$
and $h| C$ (since $C^{vWU} = A^{WU} = B^{U} = C)$, i.e. $h$ is the
perpendicular from $C$ to $s$.

Hence $W, W^{h}$ are points on $s$ with $CW \equiv CW^{h}$. Since
$AU \equiv CW$ it is $A^{v}U^{v} \equiv CW$ and hence $CU^v\equiv
CW$ with $U^{v}| s$ (since $U, v| s$). According to Lemma \ref{lmed}
there are at most two points $P$ on $s$ with $CW\equiv CP$, namely
$W$ and $W^h$. Hence $U^v=W$ or $U^v=W^h$.

If $U^v=W^h$ then $U^{v} = W^{vWU} = W^{UWv}$ and hence $U =
W^{UW}$. This implies $(UW)^{3} = 1$ which is a contradiction to our
assumption that $(G, S, P)$ is of characteristic $\neq 3$.

Hence $U^v=W$ and $B^{v} = (C^{U})^{v} = U^{v}C^{v}U^{v} = WAW = B$
which shows that $v$ is a line through $B$ with $A^{v} = C$.

\end{proof}

\section{An absolute order-free version of the Steiner-Lehmus
theorem}

In its original version, stating that a triangle with two congruent
internal bisectors must be congruent, the Steiner-Lehmus theorem
requires the notion of betweenness, to ensure that the two angle
bisectors are {\em internal}.

However, we will show that it is possible to state and prove an
order-free absolute version of the Steiner-Lehmus theorem, one
stated inside the theory of {\em metric planes}, from which all we
need are the axioms H2-H4 and ``For all $A$, $B$, with $A\neq B$,
there exists $c$ with $A, B | c$". Metric planes will be again
considered in group-theoretical terms, with $(G,S,P)$ as in Section
\ref{med}. The elements of $G$ will be again referred to as {\em
motions}.

To this end, we first notice that, for the angle bisectors of a
triangle $ABC$, we have the following facts that can be proved to
hold in metric planes:

(a) If there is an angle bisector $w$ through $A$, then there is
exactly another angle bisector $v$ through $A$, which is the
perpendicular in $A$ on $w$.

(b) Every triangle $ABC$ has precisely six angle bisectors (through
each of the points $A$, $B$, and $C$, there are precisely two
perpendicular angle bisectors

(c) If an angle bisector through $A$ intersects an angle bisector
through $B$ in a point $M$, then the line joining $M$ and $C$ is an
angle bisector through  $C$.

(d) If $f$, $g$, and $h$ are three arbitrary angle bisectors through
$A$, respectively $B$, respectively  $C$, and if $u$, $v$, and $w$
are the three remaining angle bisectors of triangle $ABC$, then
either $f$, $g$, and $h$ or else $u$, $v$, and $w$ have a point in
common.

According to (d), for three  angle bisectors $u$, $v$, and $w$
through $A$, respectively $B$, respectively  $C$ --- three vertices
of a triangle
---  there are two possible cases; (*) $u$, $v$, and $w$ have a common point, or
(**) $u$, $v$, and $w$  are the sides of a triangle (i.\ e.\ they
intersect pairwise in three non-collinear points).

With $\equiv$ defined as in Definition 1, we have

\begin{lem} If $AM\equiv  BM$ then there exists a motion $\alpha$ with
$A^{\alpha}= B$ and $M^{\alpha}= M$.
\end{lem}

\begin{proof} Suppose there is a motion $\alpha$ with $A^{\alpha} = M$ and $M^{\alpha} = B$.
Given that $M$ is the image of $A$ under a motion, $A$ and $M$ must
have, by \cite[\S3, Satz 28]{bach} a midpoint $N$. Thus
$A^{N\alpha}=M^{\alpha}=B$ and   $M^{N\alpha} = A^{\alpha}= M$.
\end{proof}

Here is now the order-free, absolute version of the Steiner-Lehmus
theorem:

\begin{theorem}
Let $ABC$ be a triangle with sides $a$, $b$, and $c$, and $u$, $v$,
and $w$ are angle bisectors through $A$, respectively $B$,
respectively  $C$. Then we have:

(a) If $u$, $v$, and $w$ have a point $M$ in common and $AM\equiv
BM$, the triangle $ABC$ is isosceles.

(b) If $u$, $v$, and $w$ are the sides of a triangle with  vertices
$U$, $V$, and $W$, and triangle $UVW$ is isosceles, then so is
triangle $ABC$.

\end{theorem}

\begin{proof} (a): By Lemma 1, we can assume that there is a motion
$\alpha\in G$ with $M^{\alpha}=M$ and $A^{\alpha}=B$. Since $A, M |
u$, the motion $u\alpha$ must also satisfy $M^{u\alpha}=M$ and
$A^{u\alpha}=B$. According to \cite[\S3.1]{bachje}, we must have
$\alpha\in S$ or $u\alpha\in S$. We conclude that there exists a
line $h$ with $h | M$ and $A^h = B$. Given the uniqueness of the
joining line of two points (i.\ e., given H2), we have $h | c$ and
$uh = v$ (the latter holds since $A,M | u$ and $B,M | v$). Since $u,
h, v | M$, we have $uhv \in S$ and $b^{uhv} = c^{hv} = c^v = a$.

We conclude that $uhv = uhu^h = uh(huh) = h$ is an angle bisector of
$a$ and $b$. By \cite{bach}, we have $h | C$, and since $A^h = B$,
 $h$ is a symmetry axis of triangle $ABC$, i.\ e.\, the latter is isosceles.

(b): Let $u, v | W$;  $u,w | V$;  $v,w | U$,  and let $UW \equiv
VW$. As in (a), one can prove that there is a line $m$ with   $m |
W$,  $U^m = V$, and $m | w$. Line $m$ joins the point $W$ of
intersection of the angle bisectors $u$ and $v$ with $C$, and thus
is (according to \cite[\S4,7, Satz 11]{bach}) an angle bisector
through $C$. The reflection in $m$ thus switches the lines $a$ and
$b$, as well as the lines $u$ and $v$. Thus it switches the
intersection points $A$ (of $b$ and $u$) and $B$ (of $a$ and $v$).
This means that $m$ is symmetry axis of triangle $ABC$, i.\ e., the
latter is isosceles.
\end{proof}

This formulation of the Steiner-Lehmus theorem also shows that there
is, indeed, a version of the Steiner-Lehmus theorem that is
invariant under what is called an `extraversion' in \cite{cr}.


\begin{thebibliography}{n}


\bibitem{ah} S. Abu-Saymeh, M. Hajja, More on the Steiner-Lehmus
theorem, J.  Geom.  Graphics \textbf{14} (2010), 127--133.

\bibitem{cas} Anonymous, Mathematical Note 1069, Math. Gazette
\textbf{17} (1933), 122--126.

\bibitem{jme} Anonymous, Troisi\`eme d\'emonstration, J. math.\ \'el\'em.\ \textbf{9} (1885), 131--132.

\bibitem{bach3} F. Bachmann, Geometrien mit euklidischer Metrik, in denen es zu jeder Geraden
durch einen nicht auf ihr liegenden Punkt mehrere Nichtschneidende
gibt I, II, III, Math.\ Z. \textbf{51} (1949), 752--768, 769--779;
Math. Nachr. \textbf{1} (1948), 258--276.

\bibitem{bach} F. Bachmann,   Aufbau der Geometrie
aus dem Spiegelungsbegriff,  Springer Verlag, Berlin, 2. Auflage,
1973.

\bibitem{bachje} F. Bachmann, Ebene Spiegelungsgeometrie,
 Bibliographisches Institut, Mannheim, 1989.

\bibitem{bk} F. Bachmann, W. Klingenberg, \"Uber Seiteneinteilungen in affinen und euklidischen
Ebenen, Math.\ Ann.\ \textbf{123} (1951). 288--301.


\bibitem{ber} P. Bernays, Bemerkungen zu den Grundlagen der Geometrie, in:
 K. O. Freidrichs (ed.), Studies and Essays Presented R. Courant on his 60th Birthday, January 8,1948, pp. 29--44. Interscience Publishers,  New York, 1948.

\bibitem{bli1} H. F. Blichfeldt, Proof of a theorem concerning isosceles
triangles, Ann.\  Math.\ \textbf{4} (1902), 22--24.

\bibitem{bli2} H. F. Blichfeldt,  Demonstrations of a pair of theorems in geometry,
Proc.\ Edinburg Math.\ Soc.\  \textbf{20} (1902), 16--17.

\bibitem{cr} J. Conway, A. Ryba, The Steiner-Lehmus angle bisector
theorem, Math.\ Gazette \textbf{98} (2014), 193--203.

\bibitem{cox69} H. S. M. Coxeter,  Introduction to geometry, 2nd ed.\ Wiley,  New
York, 1969.

\bibitem{des} Descube, Th\'eor\`eme de g\'eom\'etrie, J. math. \'el\'em. \textbf{4} (1880), 538--539.




\bibitem{haj} M. Hajja,  Other versions of the Steiner-Lehmus
theorem, Amer.\ Math.\ Monthly \textbf{108} (2001), 760--767.

\bibitem{hen0} A. Henderson, A classic problem in Euclidean
geometry. A basic study, J. Elisha Mitchell Sci.\ Soc.\ \textbf{53}
(1937), 246--281.

\bibitem{hen} A. Henderson, The Lehmus-Steiner-Terquem problem in global
survey, Scripta Math.\ \textbf{21} (1955), 223--232; 309--312.

\bibitem{hil99} D. Hilbert, Grundlagen der Geometrie,
12. Auflage, Teubner, Stuttgart, 1977.

\bibitem{h} R. W. Hogg, Equal bisectors revisited,
Math. Gazette \textbf{66} (1982), 304.

\bibitem{kh} A. Kharazishvili, Some topologic-geometrical
properties of external bisectors of a triangle, Georgian Math. J.
\textbf{19} (2012), 697--704.

\bibitem{mac1} J. S. MacKay,
History of a theorem in elementary geometry, Proc.\ Edinburg Math.\
Soc.\  \textbf{20} (1902),  18--22.

\bibitem{mac2} D. J. MacKay, The Lehmus-Steiner theorem, School
Sci.\ Math.\  \textbf{39} (1939),  561--572.

\bibitem{sl} J. A. M'Bride,
The equal internal bisectors theorem, 1840-1940. Many solutions or
none. A centenary account, Edinburgh Math.\ Notes \textbf{33}
(1943), 1--13.

\bibitem{np} V. Nicula, C. Pohoa\c{t}\u{a}, A stronger form of the Steiner-Lehmus theorem,
J. Geom.\ Graphics \textbf{13} (2009), 25--27.

\bibitem{ox} V. Oxman, Two Cevians intersecting on an angle bisector,
Math.\ Mag \textbf{85} (2012), 213--215.

\bibitem{pambjsl} V. Pambuccian,   What is the natural Euclidean metric?,
J. Symbolic Logic \textbf{59} (1994), 711.

\bibitem{pamb4} V. Pambuccian,  Ternary operations as primitive notions for constructive plane geometry. IV.
Math.\ Logic Quart.\ \textbf{40} (1994),  76--86.

\bibitem{pambbpas} V. Pambuccian,   Constructive axiomatization of non-elliptic
metric planes, Bull.\ Polish Acad.\ Sci.\ Math.\ \textbf{51}, (2003)
49--57.


\bibitem{pambss} V. Pambuccian, Orthogonality as single primitive notion for
metric planes. With an appendix by Horst and Rolf Struve, Beitr\"age
Algebra Geom.\ \textbf{48} (2007),  399--409.

\bibitem{stru} V. Pambuccian, R. Struve, On M. T. Calapso's characterization of the metric
of an absolute plane, J. Geom.\ {\bf 92} (2009), 105--116,

\bibitem{pej64} W. Pejas,  Eine algebraische Beschreibung der angeordneten Ebenen mit
nicht\-euklidischer Metrik,  Math.\ Z. \textbf{83} (1964), 434--457.

\bibitem{rus}  L. J. Russell, The ``equal bisector'' theorem,  Math.\ Gazette,
\textbf{45} (1961),  214--215.

\bibitem{sp} R. Schnabel, V. Pambuccian, Die metrisch-euklidische Geometrie als
Ausgangspunkt f\"ur die geordnet-euklidische Geometrie, Exposition.\
Math.\ \textbf{3} (1985),  285--288.

\bibitem{schr} E. M. Schr\"oder, Geometrie euklidischer Ebenen, Ferdinand Sch\"oningh,
 Paderborn, 1985.

\bibitem{sim} M. Simon, \"Uber die Entwicklung der Elementargeometrie im XIX.
Jahrhundert, Jahresber.\ Deutsch.\ Math.-Ver.\ Der
Erg\"anzungsb\"ande I. Band, B. G. Teubner, Leipzig, 1906.

\bibitem{sor} K. S\"orensen,  Ebenen mit
Kongruenz, J. Geom. \textbf{22}  (1984), 15--30.

\bibitem{sst} W. Schwabh\"auser, W. Szmielew and A. Tarski,
Metamathematische Methoden in der Geometrie, Springer Verlag,
Berlin, 1983 (reissued by Ishi Press International, 2011).

\bibitem{tar} G. Tarry,  Sur un th\'eor\`eme ind\'ependant du postulatum d'Euclide.
J. math.\ \'el\'em. (4) \textbf{4} (1895), 169--170.

\bibitem{tar2} G. Tarry, Solution de la question 100,
L'interm\'ediaire des math\'ematiciens \textbf{2} (1895), 327--328.

\bibitem{woy} H. G. Woyda, Note inspired by the Steiner-Lehmus
theorem, Math. Gazette \textbf{57} (1973), 338--339.

\bibitem{yze} J. van Yzeren, Equality of bisectors, an intriguing property.
Nieuw Arch. Wisk. (4) \textbf{15} (1997), 63--71.









\end{thebibliography}
\end{document}